\documentclass[reqno, 12pt]{amsart}
\usepackage{amsmath,amsfonts,amsthm,amscd,upref,amstext}
\usepackage{amssymb}
\usepackage{subfigure}
\usepackage[all]{xy}
\usepackage{comment}
\usepackage{xcolor}
\usepackage{caption}
\usepackage{fullpage}
\usepackage[colorinlistoftodos]{todonotes}

\newcommand{\C}{\mathbb{C}}

\newcommand{\OO}{\mathcal{O}}

\newtheorem{prop}{Proposition}[section]
\newtheorem{thm}[prop]{Theorem}
\newtheorem{cor}[prop]{Corollary}
\newtheorem{ex}[prop]{Example}
\newtheorem{defn}[prop]{Definition}
\newtheorem{question}[prop]{Question}

\newtheorem{rem}[prop]{Remark}
\newtheorem{lem}[prop]{Lemma}

\newcommand{\Z}{\mathbb{Z}}

\newcommand{\ia}{\mathfrak{a}}
\newcommand{\ib}{\mathfrak{b}}
\newcommand{\ic}{\mathfrak{c}}

\DeclareMathOperator{\Spec}{Spec}
\DeclareMathOperator{\Spf}{Spf}

\DeclareMathOperator{\edim}{edim}

\usepackage{tikz}

\usepackage{tikz-cd}
\usetikzlibrary{matrix}

\usepackage{lmodern} 
\usepackage{amsmath}
\usepackage{amssymb}
\usepackage{graphicx,scalerel}

\usepackage{scalerel,stackengine}
\stackMath

\newcommand\reallywidehat[1]{%
\savestack{\tmpbox}{\stretchto{%
  \scaleto{%
    \scalerel*[\widthof{\ensuremath{#1}}]{\kern-.8pt\bigwedge\kern-.8pt}%
    {\rule[-\textheight/2]{1ex}{\textheight}}
  }{\textheight}%
}{0.85ex}}%
\stackon[1pt]{#1}{\tmpbox}%
}

\numberwithin{equation}{section}

\begin{document}

\title{On  the gluing of formal schemes}

\date{}

\author{R. A. Calixto}
\address{Universidade de S{\~a}o Paulo -
ICMC, 13560-970, S{\~a}o Carlos-SP, Brazil}
\email{rejicalixto@usp.br}

\author{T. H. Freitas}
\address{Universidade Tecnol\'ogica Federal do Paran\'a, 85053-525, Guarapuava-PR, Brazil}
\email{freitas.thf@gmail.com}

\author{V. H. Jorge P\'erez}
\address{Universidade de S{\~a}o Paulo -
ICMC, 13560-970, S{\~a}o Carlos-SP, Brazil}
\email{vhjperez@icmc.usp.br}


 
\begin{abstract}


The main focus of this paper is to show that the gluing of formal schemes is also a formal scheme. The algebraic approach established here also leads us to conclude when the gluing of $k$-formal schemes is a $k$-formal scheme.  In addition, we derive that the gluing of formal schemes is always singular, regardless of whether we know the structure of the schemes involved.





\end{abstract}

\maketitle

\section{Introduction}

The theory of schemes is one of the cornerstones of modern algebraic geometry. By generalizing classical algebraic geometry, it provides a more comprehensive and global approach to the study of varieties, including those with singularities and other complexities. Grothendieck \cite{EGAI}, in his seminal works such as ``Éléments de géométrie algébrique'', emphasizes the importance of overcoming the limitations of the classical approach, highlighting how this broader perspective propels the advancement of the theory.

In the context of abelian groups associated with schemes, this generalization significantly extends the theoretical tools available. The use of techniques such as group cohomology in the scheme context allows for a deeper understanding of the properties of these groups, as demonstrated in Grothendieck's works. The theory of schemes not only extends but also enriches algebraic geometry, providing a more comprehensive framework to understand both algebraic varieties and the groups associated with them.

In algebraic geometry, two important subfields arise from different approaches: if $k$ is a field, when considering the completion of the polynomial ring \( k[x_1, \dots, x_n] \) with respect to the \( (x_1, \dots, x_n) \)-adic topology, we obtain the ring of formal power series \( k[[x_1, \dots, x_n]] \). This construction is expanded when considering elements \( f_i \in k[[x_1, \dots, x_n]] \) (\( 1 \leq i \leq r \)) and exploring the solutions or zeros of these elements, leading to formal spaces.

Within \( k[[x_1, \dots, x_n]] \), when focusing on convergent formal power series or analytic functions, we arrive at another ring, denoted by \( k\{x_1, \dots, x_n\} \). Similarly, when \( f_i \in k\{x_1, \dots, x_n\} \) (\( 1 \leq i \leq r \)), the zeros of these elements generate the so-called analytic spaces. In the specific case where \( k = \mathbb{C} \), this type of space is extensively studied in the theory of complex singularities. It is important to note that there is a crucial distinction between affine algebraic varieties and formal (or analytic) affine spaces. Algebraic varieties have the Zariski topology, whereas formal or analytic spaces adopt the Hausdorff topology. This difference becomes critical in the study of singularities, as the Zariski topology is coarser, while the Hausdorff topology, being finer, better reflects the characteristics of singularities.

Just as the formal and analytic approaches are distinct, modern algebraic geometry is rich with branches, particularly in the context of the theory of schemes. This vast field is replete with open problems, many of which stem from investigations in classical algebraic geometry. 

In different types of categories (for instance for the category of schemes), the following question arises: Does the gluing of three objects in the same category remain in the same category? This  difficult question has some positive answers in the literature, for instance, when the objects are analytic spaces, $\C$-formal spaces and, in a more general situation, in the category of schemes (\cite{Israel, paperformal, KS}). In modern algebraic geometry, especially within the category of schemes, works such as \cite{artin, ferrand, orlov, KS} have been significant for this subject.


However, despite these efforts, there remains a considerable gap in exploring gluing constructions in both analytic and formal contexts, and even more so within the category of schemes. This work aims to address this gap by providing new perspectives on the issues raised in Question \ref{questionCM}, with an emphasis on the theory of formal schemes.

As a refinement of the study carried out in \cite{KS, Israel, paperformal, KS}, the main focus of this work is to give positive answers for the following questions: 

\begin{question}\label{questionCM}  Let \( \mathcal{X} \), \( \mathcal{Y} \), and \( \mathcal{Z} \) be formal schemes. Let denote by \( \mathcal{X} \sqcup_{\mathcal{Z}} \mathcal{Y} \) the gluing of \( \mathcal{X} \) and \( \mathcal{Y} \) along \( \mathcal{Z} \) (see Definition \ref{copro}).
\begin{enumerate} 
\item Is the gluing \( \mathcal{X} \sqcup_{\mathcal{Z}} \mathcal{Y} \) also a formal scheme?
\item If \( \mathcal{X} \), \( \mathcal{Y} \), and \( \mathcal{Z} \) regular, what we can say about the regularity of  \( \mathcal{X} \sqcup_{\mathcal{Z}} \mathcal{Y} \)?
\end{enumerate}
\end{question}

To address the issues raised in items (1) and (2) of Question \ref{questionCM}, it is necessary to introduce some topological/algebraic concepts, such as topological rings, linearly topologized rings, admissible rings and fiber product rings. Actually, the notion of fiber product rings plays a fundamental role in this paper, mainly for the investigation of the class of admissible rings. As a biproduct of the present investigation and a key result for the approach carried out, we show that the fiber product of admissible rings is also admissible
(Theorem \ref{fibercomplete}).




We briefly describe
 the contents of the paper.  Section \ref{admissible} contains the fundamental results about topological rings, focusing on linearly topologized, admissible, and adic rings. The section includes characterizations of admissible rings, namely, we prove that the category of admissible ideals has fiber products (Theorem \ref{fibercomplete}) and their relationship with projective limits of discrete rings. We also discuss examples to illustrate these definitions.

Section \ref{formal} gives the main results concerning the gluing of formal schemes. We begin with the definitions of formal spectra and formal schemes, preparing the ground for discussions on the topology of these spaces and the notion of gluing of formal schemes.
We present our main structural result, which shows the truth of  Question \ref{questionCM} (1) (see Theorem \ref{gluingFS}). 

Section 4 shows the impact of these constructions on the noetherianess of the gluing of formal schemes (Theorem \ref{gluingFS2}) and on the gluing of $k$-formal schemes (Theorem \ref{corglueformalafi}). 
The last section is devoted to investigate the singular behavior of the gluing of formal schemes. Actually, we show that the gluing of formal schemes is always singular (see Theorem \ref{Singularformal}). This yields a complete answer for Question \ref{questionCM} (2).

\section{Preparatory facts and results}


In this section, first we recall the main concepts and results for the rest of the paper.  A key ingredient to the algebraic approach employed in this paper is the concept of fiber product rings. In order to study the gluing of formal schemes, we will show that the category of admissible rings has fiber products, which is the main result of this section.

\subsection*{Basic concepts} We say $R$ is a \textit{topological ring} if $R$ is endowed with a topology such that both addition and multiplication are continuous as maps $R \times R \to R$, where $R \times R$ has the product topology. 
We say $R$ is \textit{linearly topologized} if $0$ has a fundamental system of neighbourhoods consisting of ideals \footnote{Definition: 
If $x \in R$, then a \textit{fundamental system of neighborhoods} of $x$ is a nonempty set $\mathcal{M}$ of open neighborhoods of $x$ with the property that if $U$ is open and $x \in U$, then there is $V \in \mathcal{M}$ such that $V \subseteq U$.}. 
We say $R$ is \textit{linearly topologized} if $0$ has a fundamental system of neighborhoods consisting of ideals. For example, let \( \{\mathfrak{a}_\lambda\}_{\lambda \in \Lambda} \) be a directed system of ideals of \( R \), where \( \Lambda \) is a directed set. Then the inverse limit topology  \( \varprojlim_{\lambda \in \Lambda} R/\mathfrak{a}_\lambda \) is linearly topologized, with a fundamental system of neighborhoods of \( 0 \) given by \( \ker(R \to R/\mathfrak{a}_\lambda) \).

Let \( R \) be a linearly topologized ring, and let \( \{\mathfrak{a}_\lambda\}_{\lambda \in \Lambda} \) be a fundamental system of open ideals in \( R \). Defining \( \lambda \geq \lambda' \iff \mathfrak{a}_\lambda \subseteq \mathfrak{a}_{\lambda'} \), we see that \( \Lambda \) is a directed set. We then obtain a natural homomorphism of linearly topologized rings:
\[
c: R \to \varprojlim_{\lambda \in \Lambda} R/\mathfrak{a}_\lambda.
\]
The ring \( R \) is {\it separated } (as a topological space) if and only if \( c \) is injective. We say that \( R \) is {\it complete} if \( c \) is an isomorphism. This condition is independent of the choice of the fundamental system of open ideals \( \{\mathfrak{a}_\lambda\}_{\lambda \in \Lambda} \). Observe that the completion \( \widehat{R}: = \varprojlim_{\lambda \in \Lambda} R/\mathfrak{a}_\lambda \) is independent of this choice and is always complete.

Let $R$ be a ring endowed with the topology defined by a descending filtration ${\mathcal{F}}=\{\ia_{\lambda}\}_{\lambda\in \Lambda}$ of ideals.

    An ideal $\ia\subseteq R$ is said to be an \textit {ideal of definition} of the linearly topologized ring $R$ if the following conditions are satisfied:
    \begin{itemize}
        \item[(i)] $\ia$ is \textit {open}; that is, there exists $\lambda\in \Lambda$ such that $\ia_\lambda \subseteq \ia$;
\item[(ii)] $\ia$ is \textit{topologically nilpotent}; that is, for any $\mu\in \Lambda$ there exists $n \geq 0$ such
that $\ia^n \subseteq \ia_{\mu}$.
 \end{itemize}


\begin{defn}{\rm 
Let \( R \) be a linearly topologized ring. We say that \( R \) is \textit{admissible} if \( R \) has an ideal of definition and is complete.}
\end{defn}

If \( R \) is a linearly topologized ring, we say that \( R \) is \textit{adic} if there exists an ideal of definition \( \mathfrak{a} \) such that the set \( \{\mathfrak{a}^n \mid n \geq 0\} \) forms a fundamental system of neighborhoods of \( 0 \), and \( R \) is complete. 

Note that an adic topological ring is the same as an admissible ring with an ideal of definition \( \mathfrak{a} \) such that \( \mathfrak{a}^n \) is open for all \( n \geq 1 \). 

\begin{ex} {\rm  
The ring \( k[[x_1, \dots, x_n]] \) is a classical example of an admissible ring and also an adic ring, with the ideal of definition given by \( \mathfrak{m} = (x_1, \dots, x_n) \).  

Now, consider the ring \( R:= k[[x, y]] \), equipped with the topology defined by the family of ideals \( \{(x y^n)\}_{n \in \mathbb{N}} \). In this case, \( R \) is admissible because \( \{(x y^n)\}_{n \in \mathbb{N}} \) forms a basis of neighborhoods of \( 0 \), and each \( (x y^n) \) is an ideal. However, \( R \) is not adic, since for \( i \geq 2 \) and \( n \geq 0 \), the ideal \( (x^i y^i) \) does not contain \( (x y^n) \). Consequently, \( (x^i y^i) \) is not open in the given topology. } 
\end{ex}  
\begin{lem}{\rm (\cite[0, Lem. 7.2.2]{EGAI})}. A topological ring \( R \) is admissible if and only if \( R \) is isomorphic to the limit 
$
\varprojlim R_i,
$
of some projective system of discrete rings,
$
R_1 \leftarrow R_2 \leftarrow \cdots,
$
such that for every \( i \), the map \( R_{i+1} \to R_i \) is surjective and every element of its kernel is nilpotent.
\end{lem}

To be or not an ideal of definition $\ia \subseteq R$,  depends on the topology being considered in $R$. See the following example.
\begin{ex}{\rm
Consider the ring $R = \Z_p[x]$. We have that $R$ is a topological ring with the following two topologies:
\begin{itemize}
\item[(1)] the topology where a neighborhood basis for $0$ is given by $\mathcal{F}=\{(x)^n\}$; and
\item[(2)] the topology where a neighborhood basis for $0$ is given by $\mathcal{S}=\{(p)^n\}.$
\end{itemize}

Note that these two topologies are not the same. A easiest way to see this is to take their respective completions: the completion of the former is the ring $\Z_p[[x]]$, whereas the completion of the latter is $\Z_p\{x\}$ (the ring of restricted formal power series).

If we consider $R$ with the topology defined by $\mathcal{F}$, then $(x)$ is an ideal of definition of $R$. However, if we consider $R$ with the topology defined by $\mathcal{S}$, then $(x)$ is no longer an ideal of definition of $R$.
}
\end{ex}

The following examples illustrate rings that have an ideal of definition, their topology is $\ia$-adic, but they are not adic rings.

\begin{ex} {\rm
 \begin{itemize}
     \item[(1)] The ring $R[x]$, where a basis of neighborhoods of $0$ is given by powers of the ideal $(x)$. As before, more generally, $R[x_1,\dots,x_n]$ with a neighborhood basis of $0$ given by powers of the ideal $(x_1,\dots, x_n)$. 
\item[(2)] $\Z$ equipped with the $m$-adic topology, for any $m \in \Z$ (of course, the cases of primary
interest are when $m$ is in fact a prime $p$).
\end{itemize}
}
\end{ex}

\begin{ex} {\rm The following rings are adic:
\begin{itemize}
\item[(1)]  Any commutative ring under the discrete topology. With a neighborhood basis is $\{(0)\}$, which is also $\{(0)^n\}$.
\item[(2)] $R[[x]]$ (resp. $R[[x_1,\dots,x_n]])$, with a neighborhood basis $\{(x)^k\}$ (resp. $\{(x_1,\dots,x_n)^k\}$). This is just the completion of the first ring mentioned above.
\end{itemize}
}
\end{ex}

\begin{thm}\label{complete1}
 Let $R$ be a ring that admits an ideal of definition, and $\{\ia_{\lambda}\}_{\lambda\in \Lambda}$ be a fundamental system of ideals of definition. Then \( \varprojlim R / \ia_\lambda \) is admissible.
\end{thm}
\begin{proof}
   By \cite[Proposition 7.1.8(2)]{FK}, one has that  $\widehat{R}$  is separated, complete, and the inverse limit topology naturally makes $\widehat{R}$  a topological ring. To conclude, it is enough to show that $\hat{\ia}_{\lambda}:=\text{ker }\pi_{\lambda}$  is the ideal of definition of $\widehat{R}$.
\end{proof}

\begin{rem}\label{universalcompletion} {\rm The completion $\widehat{R}$ of a ring $R$ satisfies the universal property. In fact, the ring $\widehat{R}$, as defined above, satisfies the following:
with the previous notation, if $\varphi : R \to S$ is a  homomorphism of rings, where $S$ is an admissible ring,
 then there exists a unique continuous homomorphism $ \phi: \widehat{R}\to S$ such that the following diagram commutes:
$$
\xymatrix{
&R\ar[d]_{\varphi}\ar[r] & \widehat{R} \ar@{.>}[ld]^{{\phi}} \\
 &S&  }
$$
}
\end{rem}

\subsection*{Fiber product of admissible rings}\label{admissible}
A key ingredient for the rest of the paper is the notion of fiber product rings. Let $R$, $S$ and $T$ be
 commutative rings,  and let $R \stackrel{\pi_R}\longrightarrow T \stackrel{\pi_S}\longleftarrow S$ be homomorphisms of rings.  The fiber product ring is the set 
$$
R \times_T S=\{(r,s)\in R\times S  \ \mid \ \pi_R(r)=\pi_S(s)  \},
$$
 and it is a
subring of the usual direct product $R\times S$. In addition, if $R \stackrel{\pi_R}\longrightarrow T \stackrel{\pi_S}\longleftarrow S$  are surjective homomorphisms of rings,  \cite[Lemma 1.2]{AAM} (or \cite[Lemma 2.1]{shirogoto}) gives that  $R \times_T S$ is a Noetherian  and local ring provided $R$, $S$ and $T$ are Noetherian local rings. From now on, we assume that the fiber product $R\times_TS$ is non-trivial, i.e.,  $R\neq T\neq S$.

Let $\eta_R: R \times_T S\twoheadrightarrow R$
and $\eta_S:R \times_T S\twoheadrightarrow S$ be the natural
projections $(r,s)\mapsto r$ and $(r,s)\mapsto s$, respectively. Then $R \times_TS$ is represented as a pullback diagram:
\begin{equation}\label{dia}
\begin{CD}
R \times_T S & \stackrel{\eta_S}{\longrightarrow} & S\\
{\eta_R}{\downarrow} &  & {\downarrow}{\pi_S} \\
R & \stackrel{\pi_R}{\longrightarrow} & T.
\end{CD}
\end{equation}

\begin{rem}\label{dimdepth}{ \rm An important fact given in \cite{AAM} concerning   the fiber product ring  $R\times_TS$ is that,
$$\dim (R\times_TS) = {\rm max} \{ \dim (R), \dim (S) \},$$
$${\rm depth} (R\times_TS) \geq {\rm min} \{{\rm depth} (R), {\rm depth} (S), {\rm depth}(T)+ 1\}.$$

If $T=k$, then
$${\rm depth} (R\times_kS) = {\rm min} \{{\rm depth} (R), {\rm depth} (S), 1\} .$$
 
 }
\end{rem}

\begin{rem}\label{Union}
{\rm Now we recall some explicit forms to produce fiber products. 
\begin{enumerate}
\item[(i)] Let $R$ be a Noetherian local ring. Let $I$ and $J$ be two non-zero ideals of $R$. Note that the map $\psi: R \rightarrow R/I\times_{R/I+J} R/J$, given by $\psi(r)=(r+I,r+J)$, is surjective. Therefore  
$R/I\cap J \cong R/I\times_{R/I+J} R/J.$

\item[(ii)] Let $R = k[[x_1,\dots,x_n]]$ and $S = k[[y_1,\dots,y_m]]$ be two polynomial rings, where $k$ is a field. Let $I$ and $J$ be  ideals of $R$ and  $S$, respectively. Then 
 $$\frac{k[[x_1,\dots,x_n]]}{I}\times_k \frac{k[[y_1,\dots,y_m]]}{J}\cong \frac{k[[x_1,\dots,x_n, y_1,\dots,y_m]]}{(I+J+(x_iy_j))},$$
\end{enumerate}
}
\end{rem}



Consider  $R$, $S$ and  $T$  tree admissible rings, where $\varepsilon_R : R \rightarrow T,\varepsilon_S: S \rightarrow T$ are homomorphisms of rings. Let $\{\ia_{\lambda}\}_{\lambda\in \Lambda}$, $\{\ib_{\beta}\}_{\beta\in \Sigma}$ and $\{\ic_{\alpha}\}_{\alpha\in C}$  be a fundamental system of ideals of definition of $R$, $S$ and  $T$, respectively, such that $\ia_{\lambda}\subseteq \varepsilon_R^{-1}(\ic_{\alpha})$ and $\ib_{\beta}\subseteq\varepsilon_S^{-1}(\ic_{\alpha})$. The \textit{completed fiber product} of  $R$, $S$ and  $T$ is given by

$$\widehat{R\times_TS}=\varprojlim\limits_{\lambda, \beta, \alpha} R\times_TS/\ia_{\lambda}\times_{\ic_{\alpha}} \ib_{\beta},$$

where $\Lambda\times\Sigma\times C$ is the directed set given by
$$\Lambda\times\Sigma\times C=\{(\alpha,\beta,\alpha)\,|\,\varepsilon_R(\ia_{\lambda})\subseteq \ic_{\alpha}, \varepsilon_S(\ib_{\beta})\subseteq \ic_{\alpha}\},$$
considered with the ordering defined by

$$(\lambda,\beta, \alpha)\leq (\lambda',\beta', \alpha') \iff \lambda\leq \lambda', \beta\leq \beta' \mbox{ and } \alpha \leq \alpha'.$$ 

The ring $\widehat{R\times_TS}$ is the completion of the fiber product the ring ${R\times_TS}$ with respect to the topology defined  by the filtration $$H^{\alpha,\beta,\alpha}:=\{\ia_{\lambda}\times_{\ic_{\alpha}} \ib_{\beta}\}_{(\alpha,\beta,\alpha)\in \Lambda\times\Sigma\times C}.$$

For each $(\alpha,\beta,\alpha)\in \Lambda\times\Sigma\times C$, let $\widehat{H}^{\alpha,\beta,\alpha}$  be the closure of the image  of $H^{\alpha,\beta,\alpha}$ in $\widehat{R\times_TS}$. Then $\widehat{R\times_TS}$ is complete with respect to the topology defined by the filtration $\widehat{H}^{\alpha,\beta,\alpha}$ (see \cite[Proposition 7.1.8(2), p. 141]{FK}).

\begin{lem}\label{fiberpre-admi}
     Let $R$, $S$ and  $T$ be rings that admit ideals of definition, and let $\{\ia_{\lambda}\}_{\lambda \in \Lambda}$, $\{\ib_{\beta}\}_{\beta \in \Gamma}$ and $\{\ic_{\alpha}\}_{\alpha \in \Sigma}$ be a fundamental system of
ideals of definition of $R$, $S$ and $T$, respectively. Then the fiber product ring $P:=R\times_TS$ admits an ideal of definition, endowed with the
$\{\ia_{\lambda}\times_{\ic_{\alpha}} \ib_{\beta}\}$-topology.
\end{lem}
\begin{proof}It is sufficient  to show that, for any $\lambda,\beta, \alpha$ and $\lambda',\beta', \alpha'$,  there exists some $n$ so that $(\ia_{\lambda}\times_{\ic_{\alpha}} \ib_{\beta})^{2n}\subseteq \ia_{\lambda'}\times_{\ic_{\alpha'}} \ib_{\beta'}.$
 Then, any element of $\ia_{\lambda}\times_{\ic_{\alpha}} \ib_{\beta}$ can be written as $(r,s)\in \ia_{\lambda}\times \ib_{\beta}$ such that $\varepsilon_R(r)=\varepsilon_S(s)=t\in \ic_{\alpha}.$ So, if $n$ is such that $\ia_{\lambda}^{n}\subseteq \ia_{\lambda'}$, $\ib_{\beta}^{n}\subseteq \ib_{\beta'}$, $\ic_{\alpha}^{n}\subseteq \ic_{\alpha'}$; and $r\in \ia_{\lambda}^{n}$, $s\in \ib_{\beta}^{n}$,  $t\in \ic_{\alpha}^{n}$, then $r\in \ia_{\lambda}$, $s\in \ib_{\beta}$,  $t\in \ic_{\alpha}$. Hence,  $(\ia_{\lambda}\times_{\ic_{\alpha}} \ib_{\beta})^{2n}\subseteq \ia_{\lambda'}\times_{\ic_{\alpha'}} \ib_{\beta'}.$
\end{proof}

\begin{rem}\label{remfiberadmissible} 
{\rm    As a consequence of Lemma \ref{fiberpre-admi} and Theorem \ref{complete1}, we derive that $\widehat{R\times_TS}$ is an admissible ring. }
\end{rem}
As proposed at the beginning of the section,  one of the main results of this paper establishes that the category of admissible rings has fiber products.
\begin{thm}\label{fibercomplete}  Let $R$, $S$ and  $T$ be tree admissible rings. Then the pullback  (also called the fiber product) $R\times_TS$ is also an admissible ring.
\end{thm}
\begin{proof}
Since $\widehat{R\times_TS}$ is admissible (Remark \ref{remfiberadmissible}), and the rings $R$, $S$ and $T$ are admissible, the universal property of pullbacks gives that the following  diagram  is commutative
$$
\begin{tikzcd}
{} & \widehat{R\times_TS}
\arrow[bend right,swap]{ddl}{q_1}
\arrow[bend left]{ddr}{q_2}
\arrow[dashed]{d}[description]{\varphi} & & \\
& R\times_TS \arrow{dr}{\eta_S} \arrow{dl}[swap]{\eta_R} \\
R \arrow[swap]{dr}{\varepsilon_S} & & 
S \arrow{dl}{\varepsilon_R} \\
& T
\end{tikzcd}
 $$
Therefore, since  $R$, $S$ and $T$ are admissible and pullbacks are unique up to isomorphism, we derive that  $\widehat{R\times_TS}\cong R\times_TS \cong \widehat{R}\times_{\widehat{T}}\widehat{S}$, and this provides the desired conclusion.
\end{proof}

\begin{ex}\label{exglu}{\rm 
\begin{enumerate}
\item 
Note that \( k[[x_1, \dots, x_n]] / I \), where \( I \) is an ideal, can be an admissible ring, but this depends on the inherited topology and the ideal \( I \). In fact,
\( k[[x_1, \dots, x_n]] \) is the ring of formal power series with the topology defined by the powers of the maximal ideal \( \mathfrak{m} = (x_1, \dots, x_n) \). This topology is linear, and \( k[[x_1, \dots, x_n]] \) is complete.
The quotient \( k[[x_1, \dots, x_n]] / I \) inherits the topology from the original ring via the quotient topology. A basis of neighborhoods of \( 0 \) in the quotient is given by the ideals \( \{\overline{\mathfrak{m}}:=\mathfrak{m}^n+I \mid n \geq 0\} \). The ring \( k[[x_1, \dots, x_n]] / I \) is linearly topologized since its basis of neighborhoods of \( 0 \) consists of ideals. Furthermore, the quotient is complete because \( k[[x_1, \dots, x_n]] \) is complete, and \( I \) is closed in the original topology. Therefore, \( k[[x_1, \dots, x_n]] / I \), with the inherited topology, is admissible. 
\item  Let $ k[[x_1, \dots, x_n]] / I$ and \( k[[y_1, \dots, y_m]] / J \)  be admissible rings (by (1)). By Example \ref{Union} (ii), one has that the fiber product $\frac{k[[x_1,\dots,x_n]]}{I}\times_k \frac{k[[y_1,\dots,y_m]]}{J}$ is also admissible.

\item Let $k[[x_1, \dots, x_n]] / I$, \( k[[y_1, \dots, y_m]] / J \) and  $\frac{k[[z_1,\dots,z_{s_{i,j}}]]}{K_{i,j}}$ be admissible rings. By Theorem  \ref{fibercomplete}, the ring ${\frac{k[[x_1,\dots,x_{n_i}]]}{I_i}\times_{\frac{k[[z_1,\dots,z_{s_{i,j}}]]}{K_{i,j}}} \frac{k[[y_1,\dots,y_{m_j}]]}{J_j}}$ is admissible.  So,  the 
Cohen-Structure Theorem  yields
\begin{align*}
\begin{array}{lll}
{\frac{k[[x_1,\dots,x_{n_i}]]}{I_i}\times_{\frac{k[[z_1,\dots,z_{s_{i,j}}]]}{K_{i,j}}} \frac{k[[y_1,\dots,y_{m_j}]]}{J_j}} &\cong &\reallywidehat{\frac{k[[x_1,\dots,x_{n_i}]]}{I_i}\times_{\frac{k[[z_1,\dots,z_{s_{i,j}}]]}{K_{i,j}}} \frac{k[[y_1,\dots,y_{m_j}]]}{J_j}}\\
&\cong & k[[w_1,\ldots,w_r]]/K.
\end{array}
\end{align*}
\end{enumerate}}
\end{ex}

\section{On the gluing of formal schemes}\label{formal}

Throughout this section, all rings will be admissible, and $\{\ia_{\lambda}\}_{\lambda \in \Lambda}$ will denote a fundamental system of neighborhoods of 0 consisting of all the ideals of definition.

\begin{rem} {\rm
Let $R$ be any ring.
    \begin{enumerate}
       
         \item 
         For each $\lambda$, we have sheaf ideals $\mathcal{I}_\lambda$ on $X=\Spec(R)$ associated with the ideal of definition $\ia_{\lambda}$. We denote by $\OO_\lambda$ the sheaf induced in $\mathcal{X}$ by $\OO_{X}/\mathcal{I}_\lambda$. For each $\ia_\mu \subset\ia_\lambda$, the canonical moprphism $R/\ia_\mu \to R/\ia_\lambda$ induces a morphism of sheaves of rings $u_{\mu\lambda}: \OO_\mu \to \OO_\lambda$. Then $\{\OO_\lambda\}_{\lambda \in \Lambda}$ forms an inverse system of sheaves in the topological space $\Spec(R/\ia)$. One has $\varprojlim\limits_{\lambda}\OO_{\lambda}$ as a sheaf of rings on $\Spec(R/\ia)$. 
    \item In particular, if $R$ is $\ia$-adic, the collection $\{\ia^{n+1}\}_{n\geq 0}$ defined by an ideal of definition $\ia$ is cofinal, and thus the above sheaf coincides with the projective $\varprojlim\limits_{n}\OO_{\Spec(R/\ia^{n+1})}$.
    \end{enumerate}
}
\end{rem}

With the previous comments, we are able to define the formal schemes.

 \begin{defn}{\rm 
Let $R$ be an admissible topological ring, the {\it formal spectrum} of \( R \), denoted by \( \mathrm{Spf}(R) \), is defined as the closed subspace \( \mathcal{X} \) of \( \mathrm{Spec}(R) \) consisting of the open prime ideals of \( R \). A topologically ringed space is called a \emph{formal affine scheme} if it is isomorphic to a formal spectrum \( \mathrm{Spf}(R) = \mathcal{X} \), equipped with the sheaf of rings \( \OO_{\mathcal{X}} := \varprojlim\limits_{\lambda} \OO_{\lambda} \), as defined above. 

A \textit{formal scheme} is a topologically ringed space \( (\mathcal{X}, \OO_{\mathcal{X}}) \) that is locally isomorphic to a formal affine scheme. For simplicity, the pair \( (\mathcal{X}, \OO_{\mathcal{X}}) \) will be denoted simply by \( \mathcal{X} \).
}
\end{defn}

An \textit {open formal subscheme} of a formal scheme $\mathcal{X}$ is a formal scheme of the form  $(U,\OO_{\mathcal{X}}|_U)$, where $U$ is an open subset of the underlying topological space of $\mathcal{X}$. 
An open formal subscheme $U\subseteq \mathcal{X}$ is said to be \textit{affine open} if it is an affine formal scheme. Thus any formal scheme  $\mathcal{X}$ allows an open covering $\mathcal{X}=\bigcup_\alpha  U_{\alpha}$ consisting of affine open formal subschemes; an open covering of this form is called an \textit{affine (open) covering}.

Just as in the case of ordinary schemes, any open of a formal scheme is a formal scheme (see \cite[(\textbf{I}, 10.4.4)]{EGAI}).

\begin{ex}{\rm Let $X$ be a scheme and $Y \subset X$ a closed subscheme, defined by a quasicoherent ideal $\mathcal{I} \subset \mathcal{O}_X$. Then consider the sheaf $\mathcal{O}_Y$ obtained by restricting the projective $\varprojlim\limits_n\mathcal{O}_X/ \mathcal{I}^n$ to $Y$. It follows that $\left(Y, \mathcal{O}_Y\right)$ is a locally topologically ringed space, the desired \textit{formal completion of $X$ along $Y$}. Locally, the construction looks as follows: Let $X=\operatorname{Spec}(R)$ and assume that $\mathcal{I}$ is associated to the ideal $\mathfrak{a} \subseteq R$. Then
$$
\left(Y, \mathcal{O}_Y\right)=\operatorname{Spf}\left(\varprojlim\limits_nR / \mathfrak{a}^n\right)=\operatorname{Spf}(\hat{R})
$$
where $\hat{R}$ is the $\ia$-adic completion of $R$.
}

\end{ex}

Next, we define the complete localization that will be essential for the results presented in this section.

 Let $R$ be an admissible ring, $\{\ia_\lambda\}_{\lambda \in \Lambda}$ a fundamental system of ideals of definition, and $S \subseteq R$ a multiplicative subset. Consider the localization $S^{-1}R$ in the multiplicative set $S$ endowed with the topology defined by $\{S^{-1}\ia_\lambda\}_{\lambda\in \Lambda}$. Let $R_{\{S\}}$ denote the completion of the ring $S^{-1}R$:
 $$
 R_{\{S\}}=\varprojlim\limits_{\lambda \in \Lambda} S^{-1}R/S^{-1}\ia_\lambda.
 $$
 We call $R_{\{S\}}$ as the \textit{complete localization} of $R$ with respect to $S$. If \( f \in R \) and \( S = \{ 1, f, f^2, \dots \} \), then we denote \( R_{\{ S \}} \) by \( R_{\{ f \}} \).


\begin{rem}\label{stalk}{\rm 
    Let \( \mathfrak{p} \) be an open prime ideal of \( R \), and let \( S \) denote the multiplicatively closed set of all \( f \in R \) with \( f \notin \mathfrak{p} \) (that is, \( \mathfrak{p} \in \mathcal{D}(f):=D(f)\cap \mathcal{X} \)). There is a canonical isomorphism of rings:
    \[
    \mathcal{O}_{\mathcal{X}, \mathfrak{p}} = \varinjlim\limits_{\mathfrak{p} \in U} \Gamma(U, \mathcal{O}_{\mathcal{X}}) \cong \varinjlim\limits_{f \in S} \Gamma(\mathcal{D}(f), \mathcal{O}_{\mathcal{X}}) \cong \varinjlim\limits_{f \in S} R_{\{ f \}} = R_{\{ S \}},
    \]
    where \( U \) is an open subset of \( \mathcal{X} \). From these isomorphisms, we conclude that \( \mathcal{O}_{\mathcal{X}, \mathfrak{p}} \) is a local ring with residue field canonically isomorphic to \( k(\mathfrak{p}) := R_{\mathfrak{p}} / \mathfrak{p} R_{\mathfrak{p}} \). Therefore, we have in particular that the ringed space \( \left( \mathcal{X}, \mathcal{O}_{\mathcal{X}} \right) \) is a locally ringed space.}
\end{rem}



 
 
\begin{defn}{\rm Given two formal schemes $\mathcal{X}$ and $\mathcal{Y}$, a \textit{morphism (of formal schemes)} from $\mathcal{X}$ to $\mathcal{Y}$ is a morphism $(\Phi, \Tilde{\Phi})$ of topologically ringed spaces such that, for all $x \in \mathcal{X}$, $\Tilde{\Phi}^{\#}_x$ is a local homomorphism $\OO_{\mathcal{Y},\Phi(x)}\longrightarrow \OO_{\mathcal{X},x}$.
}
\end{defn}


Let \( \mathcal{X} \) and \( \mathcal{Y} \) be two formal schemes. A morphism of formal schemes \( \Phi: \mathcal{X} \to \mathcal{Y} \) is an isomorphism if for each point \( x \in \mathcal{X} \), the induced map on the stalks of the structure sheaves is an isomorphism, i.e.,  if it induces a bijection of stalks of the structure sheaves at each point, and the inverse morphism \( \Psi: \mathcal{Y} \to \mathcal{X} \) also induces a bijection of stalks.

\begin{rem}\label{assumption} {\rm 
In the context of the above definition, in what follows in this paper, the formal schemes \( \mathcal{X} \) and \( \mathcal{Z} \), as well as \( \mathcal{Y} \) and \( \mathcal{Z} \), will not be isomorphic in the stalks we are considering. Otherwise, we have a trivial fiber product of 
$\mathcal{O}_{\mathcal{X},x}\times_{\mathcal{O}_{\mathcal{Z},z}}\mathcal{O}_{\mathcal{Y},y}$ which is not the case (see definition of fiber product rings for details).
}
\end{rem}

 \begin{rem}\label{affine} {\rm
As in the case of ordinary schemes, the functor
 $$
 R\longmapsto \Spf({R})
 $$
 gives rise to a categorical equivalence between the opposite category of the category of admissible rings and the  category of affine formal schemes. For additional properties regarding formal schemes, see \cite[(\textbf{I}, 10.1.3), (\textbf{I}, 10.1.4), (\textbf{I}, 10.1.6)]{EGAI} and \cite{notes}.
}
\end{rem}

\subsection*{Gluing of formal schemes}
With this new notion, it is natural to investigate the behavior of the structure of the gluing of formal schemes. For the next definition, \( \mathcal{X} \coprod \mathcal{Y} \) denotes the co-product or disjoint union of sets \( \mathcal{X} \) and \( \mathcal{Y} \).

\begin{defn}\label{copro}{\rm
Let \( \alpha: \mathcal{Z} \to \mathcal{X} \) and \( \beta: \mathcal{Z} \to \mathcal{Y} \) be morphisms of ringed spaces. Define
\[
\mathcal{X} \sqcup_{\mathcal{Z}} \mathcal{Y} = \mathcal{X} \coprod \mathcal{Y} / \sim,
\]
\noindent where the relation \( \sim \) is generated by relations of the form \( x \sim y \) (\( x \in \mathcal{X} \), \( y \in \mathcal{Y} \)) if there exists \( z \in \mathcal{Z} \) such that \( \alpha(z) = x \) and \( \beta(z) = y \).

Namely, it is the smallest equivalence relation on \( \mathcal{X} \coprod \mathcal{Y} \) such that after passing to the quotient \( \mathcal{X} \coprod \mathcal{Y} / \sim \) the following square becomes commutative:
\[
\xymatrix{\mathcal{Z} \ar[r]^{\alpha}  \ar[d]_{\beta} & \mathcal{X}  \ar[d]^{f} \\
            \mathcal{Y}  \ar[r]^{g}        & \mathcal{X} \sqcup_{\mathcal{Z}} \mathcal{Y}, }
\]
\noindent where \( f \) and \( g \) are the continuous natural maps.}
\end{defn} 

Since \( (\mathcal{X}, \mathcal{O}_{\mathcal{X}}) \), \( (\mathcal{Y}, \mathcal{O}_{\mathcal{Y}}) \), and \( (\mathcal{Z}, \mathcal{O}_{\mathcal{Z}}) \) are ringed spaces, \cite[Proposition 2.2]{KS} provides that \( (\mathcal{X} \sqcup_{\mathcal{Z}} \mathcal{Y}, \mathcal{O}_{\mathcal{X} \sqcup_{\mathcal{Z}} \mathcal{Y}}) \) is also a ringed space. Therefore, \( f \) and \( g \) become morphisms of ringed spaces. Note that this definition also satisfies the universal property (\cite[Theorem 2.3]{KS}).

 \begin{prop}\label{glueformalafi}
    Let $\mathcal{X}$, $\mathcal{Y}$ and $\mathcal{Z}$ be affine formal schemes, such that $\Tilde{\alpha} : \OO_{\mathcal{X}} \longrightarrow \alpha_{\ast}\OO_{\mathcal{Z}}$ and $\Tilde{\beta}: \OO_{\mathcal{Y}} \longrightarrow \beta_{\ast}\OO_{\mathcal{Z}}$ are both surjective homomorphisms. Then,
$\mathcal{X}\sqcup_{\mathcal{Z}}\mathcal{Y}$ is an affine formal scheme.
\end{prop}
\begin{proof}
 By Theorem \ref{fibercomplete}, the category of admissible rings admits pullback. In addition, since the category of affine formal schemes is the dual category of admissible rings (Remark \ref{affine}), one has 
 \begin{alignat*}{2}
     \mathcal{X}\sqcup_{\mathcal{Z}}\mathcal{Y}=\Spf( R\times_TS).\tag*{\qedhere}
 \end{alignat*}
\end{proof}

We are able to show the main result of this section.

\begin{thm}\label{gluingFS}
 Let $\mathcal{X}$, $\mathcal{Y}$ and $\mathcal{Z}$ be formal schemes such that $\alpha : \mathcal{Z} \longrightarrow \mathcal{X}$ and $\beta : \mathcal{Z} \longrightarrow \mathcal{Y}$ are homeomorphism onto its image. Then, $\mathcal{X}\sqcup_{\mathcal{Z}}\mathcal{Y}$ is a formal scheme.
  
 
\end{thm}
\begin{proof}
In order to give the structure sheaf on $\mathcal{W}:=\mathcal{X}\sqcup_{\mathcal{Z}}\mathcal{Y}$, one canonical way is the following. First we define a  presheaf $\mathcal{F}$ on $\mathcal{W}$ in the following way:
\begin{itemize}
    \item[(i)] For $p\in \mathcal{W} \setminus \mathcal{Y}$, put $\mathcal{F}_p:=\OO_{\mathcal{X},p}.$

    \item[(ii)] For $p\in \mathcal{W} \setminus \mathcal{X}$, put $\mathcal{F}_p:=\OO_{\mathcal{Y},p}.$
    
    \item[(iii)] For $p=\alpha(z)=\beta(z)$, put $\mathcal{F}_p:=\OO_{\mathcal{X},\alpha(z)}\times_{\OO_{\mathcal{Z},z}}\OO_{\mathcal{Y},\beta(z)}.$
\end{itemize}

Now, for an open subset $U \subset \mathcal{W}$ (an open subset in $\mathcal{W}$ is defined as the subset whose pull-back to
$\mathcal{X}\sqcup_{\mathcal{Z}}\mathcal{Y}$ is an open subset), we define
$$\mathcal{F}(U):=\prod_{p\in U}\mathcal{F}_p.$$

For open subsets $V\subset U$ of $\mathcal{W}$, define the restriction map $\rho_V^U:\mathcal{F}(U)\to \mathcal{F}(V)$ in the canonical way. Then $\mathcal{F}$ be come a presheaf on $W$. Let $\OO_\mathcal{W}$ be the sheafication  $\mathcal{F}$. 

We cover $\mathcal{X}$, $\mathcal{Y}$ and $\mathcal{Z}$ by affine formal schemes, respectively. In fact, by \cite[Lemma 2.4]{KS}, 
we get the open  subsets  $U_i\sqcup_{W_{i,j}}V_j$ of $\mathcal{X}\sqcup_{\mathcal{Z}}\mathcal{Y}$ correspond bijectively to open  subsets $U_i$ and $V_j$ of the $\mathcal{X}$ and $\mathcal{Y}$, respectively, where $W_{i,j}:=\alpha^{-1}(U_i)=\beta^{-1}(V_j).$ Note that $W_{i,j}$ is a formal affine open, and their union covers the scheme $\mathcal{Z}$ because the maps  $\alpha: \mathcal{Z} \to \mathcal{X}$ and $\beta: \mathcal{Z} \to \mathcal{Y}$ are homeomorphism onto its image. 
 Thus, since $\mathcal{X}=\cup_{i}U_i$, $\mathcal{Y}=\cup_{j}V_j$ and $\mathcal{Z}=\cup_{i,j}W_{i,j}$ are such that $U_i=\Spf(R_i)$,  $V_j=\Spf(S_j)$ and $W_{i,j}=\Spf(T_{i,j})$, where $R_i$, $S_j$, $T_{i,j}$ are admissible rings. So, we have a cover  $\mathcal{X}\sqcup_\mathcal{Z}\mathcal{Y}=\cup_{i,j}(U_i\sqcup_{W_{i,j}}V_j)$, i.e., it is a union of open subsets such that
$$\begin{array}{lll}
U_i\sqcup_{W_{i,j}}V_j&\cong& \Spf(R_i)\sqcup_{\Spf(T_{i,j})}\Spf(S_j)\\
&\cong & \Spf(R_i\times_{T_{i,j}}S_j), \,\,\,\, \mbox{ by Theorem \ref{fibercomplete}. }
\end{array}
$$
Therefore, the desired result follows by Proposition \ref{glueformalafi}.
\end{proof}

\section{Structural results of the gluing of formal schemes}

 A formal scheme $\mathcal{X}$ is \textit{adic} if there exists a cover of $\mathcal{X}$ by affine open formal schemes $U_i = \Spf(R_i)$, where $R_i$ is adic.
A formal scheme $\mathcal{X}$ is \textit{locally Noetherian} if  $R_i$ are Noetherian and adic. Also, $\mathcal{X}$ is \textit{Noetherian} if it is locally Noetherian and quasi-compact.

\begin{prop}\cite[(\textbf{I}, 10.6.5)]{EGAI}\label{noetherianafffine} Let $R$ be an admissible ring. Then $\Spf(R)$ is
Noetherian if and only if $R$ is Noetherian and adic.
\end{prop} 

With the previous information, we show the following result.

\begin{thm}\label{gluingFS2}
 Let $\mathcal{X}$, $\mathcal{Y}$ and $\mathcal{Z}$ be Noetherian formal schemes such that  $\alpha: \mathcal{Z}\to \mathcal{X}$ and $\beta:\mathcal{Z}\to \mathcal{Y}$ are closed immersion.  Then, $\mathcal{X}\sqcup_{\mathcal{Z}}\mathcal{Y}$ is Noetherian formal scheme.
\end{thm}


\begin{proof}
    Assuming the notation given in  Theorem \ref{gluingFS}, first the gluing of formal schemes is also a formal scheme. So for each $i,j$, it sufficient to show that $\OO_{\mathcal{X}\sqcup_{\mathcal{Z}}\mathcal{Y}}(U_i\sqcup_{W_{i,j}}V_j)$ is a Noetherian adic ring. Since $W_{i,j}:=\alpha^{-1}(U_i)=\beta^{-1}(V_j)$ is Noetherian formal affine open subsets of $\mathcal{Z}$, set $W_{i,j}=\Spf(T_{i,j})$.  Now, Propositon \ref{noetherianafffine} gives that  $T_{i,j}$ is Noetherian adic ring, as well as $R_i$ and $S_j$. By \cite[Lemma 2.1]{shirogoto}  we derive that $ R_i\times_{T_{i,j}}S_j$ is Noetherian adic ring. Since 
    $$
    U_i\sqcup_{W_{i,j}}V_j\cong \Spf(R_i\times_{T_{i,j}}S_j),
    $$
    the desired conclusion follows.
\end{proof}
It should be noted that the hypothesis imposed in the previous result is necessary in order to deduce the noetherianess of the gluing of formal schemes.
\begin{ex}\label{exnaonoeth}
{\rm Similarly to  \cite[Example 3.7]{KS},
If $k$ is a field, consider the Noetherian formal schemes $\mathcal{X}=\Spf(k[[x,y]])$, $\mathcal{Y}=\Spf(k)$ and $\mathcal{Z}=\Spf(k[[x,y]]/(x))$. Let $\OO_\mathcal{X}(\mathcal{X}) \twoheadrightarrow \OO_\mathcal{Z}(\mathcal{Z})$ and $\OO_\mathcal{Y}(\mathcal{Y}) \hookrightarrow \OO_\mathcal{Z}(\mathcal{Z})$ be two morphism of  rings. Note that the gluing $\mathcal{X}\sqcup_{\mathcal{Z}}\mathcal{Y}$ is not a Noetherian formal scheme, because $\mathcal{X}\sqcup_{\mathcal{Z}}\mathcal{Y}=\Spf(k[[x,xy,xy^2,xy^3,\dots]])$.
}
\end{ex}

\subsection*{The gluing of \(k\)-formal schemes}

\begin{defn}\label{k-fscheme} {\rm Let $k$ be an arbitrary field. We call $\mathcal{X}$ a $k$-formal scheme \textit{locally of finite type} or that $\mathcal{X}$ is locally of finite type over $k$, if there is
an affine formal open cover $\mathcal{X} =\cup_{i\in I}U_i$ such that for all $i$, $U_i=\Spf(R_i)$, $R_i$ is isomorphic to a quotient of a power series ring over the field with a suitable ideal. We say that $\mathcal{X}$ is of \textit {finite type over $k$} if $\mathcal{X}$ is locally of finite type and quasi-compact.
}
\end{defn}

\begin{question}\label{openproblem2}
{\rm Is the gluing of  $k$-formal schemes (finite type or locally of finite type)   a $k$-formal scheme (finite type or locally of finite type)?

}
\end{question}

Let $\mathcal{X}=\Spf(R)$, $\mathcal{Y}=\Spf(S)$ and $\mathcal{Z}=\Spf(T)$, where $(R,\mathfrak{m})$, $(S,\mathfrak{n})$ and $(T,\mathfrak{t})$ are $\mathfrak{m}$, $\mathfrak{n}$ and $\mathfrak{t}$-adic local Noetherian rings with the same residue field $k$, respectively. In addition, assume that $R \stackrel{\pi_R}\longrightarrow T \stackrel{\pi_S}\longleftarrow S$  are surjective homomorphisms of rings.

\begin{thm}\label{corglueformalafi}
    Consider $\mathcal{X}$, $\mathcal{Y}$ and $\mathcal{Z}$ as above.  Then,
$\mathcal{X}\sqcup_{\mathcal{Z}}\mathcal{Y}$ is an $k$-affine formal scheme of finite type.
\end{thm}
\begin{proof} By Proposition \ref{glueformalafi}, one has
$\mathcal{X}\sqcup_{\mathcal{Z}}\mathcal{Y}=\Spf( R\times_TS).$ The assumption that the maps $R \stackrel{\pi_R}\longrightarrow T \stackrel{\pi_S}\longleftarrow S$  are surjective homomorphisms of Noetherian local rings, \cite[Lemma 1.2]{AAM} (or \cite[Lemma 2.1]{shirogoto}) provides that  $R \times_T S$ is Noetherian and local ring with maximal ideal $\mathfrak{m}\times_{\mathfrak{t}}\mathfrak{n}$. Also, since $R\times_TS$ is the $\mathfrak{m}\times_{\mathfrak{t}}\mathfrak{n}$-adic and admissible (Theorem \ref{fibercomplete}), the  Cohen-Structure Theorem gives that  $R\times_TS$ is isomorphic to a quotient of a power series ring over the field $k$  with a suitable ideal.
\end{proof}

As a consequence of the previous result (or Example \ref{exglu} (3)), one obtains the following:
\begin{cor}
Let \( \mathcal{X} \), \( \mathcal{Y} \), and \( \mathcal{Z} \) be locally formal \( k \)-schemes such that \( \mathcal{Z} \to \mathcal{X} \) and \( \mathcal{Z} \to \mathcal{Y} \) are closed immersions. Then, \( \mathcal{X} \sqcup_{\mathcal{Z}} \mathcal{Y} \) is a locally formal \( k \)-scheme.
\end{cor}

\section{Singular behavior of the gluing  of formal schemes}

The next lemma is an important result for the rest of this work. For the rest of this section, we are assuming that all formal schemes are locally Noetherian.
\begin{lem}\label{stalklemma}
Let \( \mathcal{X} \), \( \mathcal{Y} \), and \( \mathcal{Z} \) be formal schemes such that \( \alpha: \mathcal{Z} \to \mathcal{X} \) and \( \beta: \mathcal{Z} \to \mathcal{Y} \) are closed immersions of formal schemes. Let \( z \in \mathcal{Z} \) and \( w \in \mathcal{X} \sqcup_{\mathcal{Z}} \mathcal{Y} \) such that \( w = \alpha(z) \sqcup_z \beta(z) \). Then,
\[
\mathcal{O}_{\mathcal{X} \sqcup_{\mathcal{Z}} \mathcal{Y}, w} \cong \mathcal{O}_{\mathcal{X}, \alpha(z)} \times_{\mathcal{O}_{\mathcal{Z}, z}} \mathcal{O}_{\mathcal{Y}, \beta(z)}.
\]
\end{lem}

\begin{proof}
Since the stalk is a local property, it is sufficient to verify the statement in the affine formal case. Let \( \mathcal{X} = \Spf(R) \), \( \mathcal{Y} = \Spf(S) \), and \( \mathcal{Z} = \Spf(T) \), where \( R \), \( S \), and \( T \) are commutative admissible ring . By the diagram 
\[
\xymatrix{R \times_T S \ar[r]^-{\pi_R}\ar[d]_{\pi_S} & R  \ar[d]^{\varepsilon_R}\\
            S  \ar[r]^{\varepsilon_S} & T, }
\]
\noindent with surjective arrows, we have that each element \( h \in R \times_T S \) corresponds to elements \( f \in R \) and \( g \in S \) such that \( \varepsilon_R(f) = \varepsilon_S(g) =: t \). By Remark \ref{stalk},  as \( (R \times_T S)_{\{h\}} \)-modules, we obtain \( S_h = S_g \), \( R_h = R_f \), and \( T_{\varepsilon_R(f)} = T_{\varepsilon_S(g)} = T_h \). Let \( z = [\mathfrak{p}_T] \in \mathcal{Z} \), \( \alpha(z) = [\mathfrak{p}_R] \in \mathcal{X} \), and \( \beta(z) = [\mathfrak{p}_S] \in \mathcal{Y} \) such that \( w = \alpha(z) \sqcup_z \beta(z) \). Then,
\[
\varinjlim_{h \notin \mathfrak{p}_R \sqcup_{\mathfrak{p}_T} \mathfrak{p}_S} R_{\{h\}} = \varinjlim_{f \notin \mathfrak{p}_R} R_{\{f\}}, \quad \varinjlim_{h \notin \mathfrak{p}_R \sqcup_{\mathfrak{p}_T} \mathfrak{p}_S} S_{\{h\}} = \varinjlim_{g \notin \mathfrak{p}_S} S_{\{g\}}\,\,{\rm  
and 
}
\varinjlim_{h \notin \mathfrak{p}_R \sqcup_{\mathfrak{p}_T} \mathfrak{p}_S} T_{\{h\}} = \varinjlim_{t \notin \mathfrak{p}_T} T_{\{t\}}.
\]
Again, by Remark \ref{stalk} and the exactness of direct limits, we obtain the following exact sequence:
\[
0 \longrightarrow \mathcal{O}_{\mathcal{X} \sqcup_{\mathcal{Z}} \mathcal{Y}, w} \longrightarrow \mathcal{O}_{\mathcal{X}, \alpha(z)} \oplus \mathcal{O}_{\mathcal{Y}, \beta(z)} \longrightarrow \mathcal{O}_{\mathcal{Z}, z} \longrightarrow 0.
\]
This yields the desired result.
\end{proof}

For a Noetherian local ring $(R, \mathfrak{m})$, the minimal number of generators of $\mathfrak{m}$ will be denoted by $\edim(R):=\dim_{\C}\mathfrak{m}/\mathfrak{m}^2$ and is
called the embedding dimension of $R$. Recall that $\edim(R) \geq \dim(R)$. If the equality happens, then $R$ is called a regular local ring.
\medskip 

Recall that a locally Noetherian formal scheme $\mathcal{X}$ is said to be  \textit{nonsingular} (or \textit{regular}) at $x\in \mathcal{X}$  if  $\OO_{\mathcal{X},x}$ is a regular ring, otherwise, we say that $\mathcal{X}$ is \textit{singular at $x$}. We say that $\mathcal{X}$ is \textit{regular} if it is regular at all of its points.

In order to show the main result of this section, first we summarize some algebraic tools.

\begin{defn}{\rm  Let $M$ be a finitely generated $\OO_{\mathcal{X},x}$-module. The Poincar\'e series of $M$ is given by
$$P_{M}^{\OO_{\mathcal{X},x}}(t):= \sum_{i\geq 0}\dim_{k(x)} {\rm Tor}_i^{\OO_{\mathcal{X},x}}\left(M, k(x) \right) t^i,$$
where $k(x):= \frac{\OO_{\mathcal{X},x}}{\mathfrak{m}_{X,x}}$ is the residue field. The number $\beta_i^{\OO_{
\mathcal{X},x}}(M):=\dim_{k(x)} {\rm Tor}_i^{\OO_{
\mathcal{X},x}}\left(M, k(x)\right)$ is called $i$-th Betti number of $M$.   

}

\end{defn}


\begin{rem}\label{rem3.3}{\rm
Let $M$ be a  $\OO_{\mathcal{X},x}$-module  with minimal number of generators $\mu(M)$. Then 
$$P^{\OO_{\mathcal{X},x}}_M(t)=\mu(M)+tP^{\OO_{\mathcal{X},x}}_{\Omega_1}(t),$$
 where $\Omega_1$  denotes the first syzygy of $M$ (see \cite{DK75}).

}
\end{rem}

\begin{rem}\label{remlev} {\rm  
\begin{itemize}
\item[(i)]
Let $F(t) = \sum
a_it^i$ and $G(t) = \sum
b_it^i$ be two power series in $t$. We say that $F(t)\succeq G(t)$ provided  $a_i \geq b_i$ for all $i$.

\item[(ii)] Let $M$ be any $\OO_{\mathcal{Y},y}$-module. By \cite[Proof of Theorem 1.1]{levin},  {\rm if $f: \OO_{\mathcal{X},x}\to \OO_{\mathcal{Y},y}$ is a surjective homomorphism, one has the spectral sequence
$${\rm Tor}_p^{\OO_{\mathcal{Y},y}}(M,k_y) \otimes {\rm Tor}_q^{\OO_{\mathcal{X},x}}(\OO_{\mathcal{Y},y},k_y) \Rightarrow {\rm Tor}_{p+q}^{\OO_{\mathcal{X},x}}(M,k_x).$$ Therefore, one obtains that 
$$P_{M}^{{\OO_{\mathcal{X},x}}}\succeq P_{M}^{{\OO_{\mathcal{Y},y}}}P_{{\OO_{\mathcal{Y},y}}}^{{\OO_{\mathcal{X},x}}}.$$

}
\end{itemize}
}
\end{rem}


We are able to show the main result of this section.

\begin{thm}  \label{Singularformal} Let $\mathcal{X}$, $\mathcal{Y}$ and $\mathcal{Z}$ be formal schemes such that \( \alpha: \mathcal{Z} \to \mathcal{X} \) and \( \beta: \mathcal{Z} \to \mathcal{Y} \) are closed immersions of formal schemes. Then, there exists a singular point $w \in  \mathcal{X}\sqcup_{\mathcal{Z}}\mathcal{Y}$.
\end{thm} 


\begin{proof} By contradiction, suppose that $\mathcal{X}\sqcup_{\mathcal{Z}}\mathcal{Y}$ is not singular for all point. Let \( z \in \mathcal{Z} \) and \( w \in \mathcal{X} \sqcup_{\mathcal{Z}} \mathcal{Y} \) such that \( w = \alpha(z) \sqcup_z \beta(z) \). In particular, $\mathcal{X}\sqcup_{\mathcal{Z}}\mathcal{Y}$ is not singular in $w$, i.e., $\mathcal{O}_{\mathcal{X} \sqcup_{\mathcal{Z}} \mathcal{Y}, w}$ is a regular ring. By Lemma \ref{stalklemma}, one has 
$$\mathcal{O}_{\mathcal{X} \sqcup_{\mathcal{Z}} \mathcal{Y}, w} \cong \mathcal{O}_{\mathcal{X}, \alpha(z)} \times_{\mathcal{O}_{\mathcal{Z}, z}} \mathcal{O}_{\mathcal{Y}, \beta(z)}.$$
By Remark \ref{dimdepth}, without loss of generality  we may assume  that $\dim \mathcal{O}_{\mathcal{X} \sqcup_{\mathcal{Z}} \mathcal{Y}, w} =\dim \mathcal{O}_{\mathcal{X}, \alpha(z)}$.
Hence, since  $\mathcal{O}_{\mathcal{X} \sqcup_{\mathcal{Z}} \mathcal{Y}, w}$ is a regular ring, we have that 
\begin{equation}\label{eqth31}
   \dim \mathcal{O}_{\mathcal{X}, \alpha(z)}= \dim \mathcal{O}_{\mathcal{X}, \alpha(z)} \times_{\mathcal{O}_{\mathcal{Z}, z}} \mathcal{O}_{\mathcal{Y}, \beta(z)} =\edim\mathcal{O}_{\mathcal{X}, \alpha(z)} \times_{\mathcal{O}_{\mathcal{Z}, z}} \mathcal{O}_{\mathcal{Y}, \beta(z)}.
\end{equation}

From the exact sequence
\begin{equation}\label{th310}
    0\longrightarrow \ker\beta_z^{\ast}\longrightarrow \mathcal{O}_{\mathcal{X}, \alpha(z)} \times_{\mathcal{O}_{\mathcal{Z}, z}} \mathcal{O}_{\mathcal{Y}, \beta(z)}\overset{\pi_1}{\longrightarrow} \mathcal{O}_{\mathcal{X}, \alpha(z)}\longrightarrow 0,
\end{equation}
and Remark \ref{rem3.3}, one has
    \begin{equation}\label{th3200}
    P_{\mathcal{O}_{\mathcal{X}, \alpha(z)}}^{\mathcal{O}_{\mathcal{X}, \alpha(z)} \times_{\mathcal{O}_{\mathcal{Z}, z}} \mathcal{O}_{\mathcal{Y}, \beta(z)}}(t)=1+tP^{\mathcal{O}_{\mathcal{X}, \alpha(z)} \times_{\mathcal{O}_{\mathcal{Z}, z}} \mathcal{O}_{\mathcal{Y}, \beta(z)}}_{\ker{\beta_z^{\ast}}}(t).
\end{equation}
The surjectivity of the map $\pi_2:\mathcal{O}_{\mathcal{X}, \alpha(z)} \times_{\mathcal{O}_{\mathcal{Z}, z}} \mathcal{O}_{\mathcal{Y}, \beta(z)}\longrightarrow \mathcal{O}_{\mathcal{X}, \alpha(z)}$ and the Remark \ref{remlev} give
\begin{equation}\label{th3300}
    tP^{\mathcal{O}_{\mathcal{X}, \alpha(z)} \times_{\mathcal{O}_{\mathcal{Z}, z}} \mathcal{O}_{\mathcal{Y}, \beta(z)}}_{\ker\beta_z^{\ast}}(t) \succeq tP_{\ker{\beta_z^{\ast}}}^{\mathcal{O}_{\mathcal{Y}, \beta(z)}}(t)P_{\mathcal{O}_{\mathcal{Y}, \beta(z)}}^{\mathcal{O}_{\mathcal{X}, \alpha(z)} \times_{\mathcal{O}_{\mathcal{Z}, z}} \mathcal{O}_{\mathcal{Y}, \beta(z)}}(t)=(P_{\mathcal{O}_{\mathcal{Z}, z}}^{\mathcal{O}_{\mathcal{Y}, \beta(z)}}(t)-1)P^{\mathcal{O}_{\mathcal{X}, \alpha(z)} \times_{\mathcal{O}_{\mathcal{Z}, z}} \mathcal{O}_{\mathcal{Y}, \beta(z)}}_{\mathcal{O}_{\mathcal{Y}, \beta(z)}}(t).
\end{equation}
Therefore,
\begin{equation}\label{th340}
    P_{\mathcal{O}_{\mathcal{X}, \alpha(z)}}^{\mathcal{O}_{\mathcal{X}, \alpha(z)} \times_{\mathcal{O}_{\mathcal{Z}, z}} \mathcal{O}_{\mathcal{Y}, \beta(z)}}(t)\succeq 1+(P_{\mathcal{O}_{\mathcal{Z}, z}}^{\mathcal{O}_{\mathcal{Y}, \beta(z)}}(t)-1)P^{\mathcal{O}_{\mathcal{X}, \alpha(z)} \times_{\mathcal{O}_{\mathcal{Z}, z}} \mathcal{O}_{\mathcal{Y}, \beta(z)}}_{\mathcal{O}_{\mathcal{Y}, \beta(z)}}(t).
\end{equation}
A similar argument with the surjective map  $\pi_1:\mathcal{O}_{\mathcal{X}, \alpha(z)} \times_{\mathcal{O}_{\mathcal{Z}, z}} \mathcal{O}_{\mathcal{Y}, \beta(z)}\longrightarrow \mathcal{O}_{\mathcal{X}, \alpha(z)}$ provides
\begin{equation}\label{th3300}
 P_{\mathcal{O}_{\mathcal{Y}, \beta(z)}}^{\mathcal{O}_{\mathcal{X}, \alpha(z)} \times_{\mathcal{O}_{\mathcal{Z}, z}} \mathcal{O}_{\mathcal{Y}, \beta(z)}}(t) \succeq  1+(P_{\mathcal{O}_{\mathcal{Z}, z}}^{\mathcal{O}_{\mathcal{X}, \alpha(z)}}(t)-1)P^{\mathcal{O}_{\mathcal{X}, \alpha(z)} \times_{\mathcal{O}_{\mathcal{Z}, z}} \mathcal{O}_{\mathcal{Y}, \beta(z)}}_{\mathcal{O}_{\mathcal{X}, \alpha(z)}}(t).
\end{equation}
Replacing   (\ref{th3300}) in   (\ref{th340}) one obtains
\begin{equation}\label{eqth32}
P_{\mathcal{O}_{\mathcal{X}, \alpha(z)}}^{\mathcal{O}_{\mathcal{X}, \alpha(z)} \times_{\mathcal{O}_{\mathcal{Z}, z}} \mathcal{O}_{\mathcal{Y}, \beta(z)}}(t)\succeq\frac{P_{\mathcal{O}_{\mathcal{Z}, z}}^{\mathcal{O}_{\mathcal{Y}, \beta(z)}}(t)}{P_{\mathcal{O}_{\mathcal{Z}, z}}^{\mathcal{O}_{\mathcal{X}, \alpha(z)}}(t)+P_{\mathcal{O}_{\mathcal{Z}, z}}^{\mathcal{O}_{\mathcal{Y}, \beta(z)}}(t)-P_{\mathcal{O}_{\mathcal{Z}, z}}^{\mathcal{O}_{\mathcal{X}, \alpha(z)}}(t)P_{\mathcal{O}_{\mathcal{Z}, z}}^{\mathcal{O}_{\mathcal{Y}, \beta(z)}}(t)}.
\end{equation}
By Remark \ref{remlev} (ii),   $P^{\mathcal{O}_{\mathcal{X}, \alpha(z)} \times_{\mathcal{O}_{\mathcal{Z}, z}} \mathcal{O}_{\mathcal{Y}, \beta(z)}}_{M}(t)\succeq P^{\mathcal{O}_{\mathcal{X}, \alpha(z)}}_{M}(t)P^{\mathcal{O}_{\mathcal{X}, \alpha(z)} \times_{\mathcal{O}_{\mathcal{Z}, z}} \mathcal{O}_{\mathcal{Y}, \beta(z)}}_{\mathcal{O}_{\mathcal{X}, \alpha(z)}}(t)$ for any $\mathcal{O}_{\mathcal{X}, \alpha(z)}$-module $M$. Now, multiplying both sides of equation (\ref{eqth32}) by $P^{\mathcal{O}_{\mathcal{X}, \alpha(z)}}_{M}(t)$ we derive

$$P^{\mathcal{O}_{\mathcal{X}, \alpha(z)} \times_{\mathcal{O}_{\mathcal{Z}, z}} \mathcal{O}_{\mathcal{Y}, \beta(z)}}_{M}(t) \succeq \frac{P^{\mathcal{O}_{\mathcal{X}, \alpha(z)}}_{M}(t)P^{\mathcal{O}_{\mathcal{Y}, \beta(z)}}_{\mathcal{O}_{\mathcal{Z}, z}}(t)}
{P^{\mathcal{O}_{\mathcal{X}, \alpha(z)}}_{\mathcal{O}_{\mathcal{Z}, z}}(t)+P^{\mathcal{O}_{\mathcal{Y}, \beta(z)}}_{\mathcal{O}_{\mathcal{Z}, z}}(t)-P^{\mathcal{O}_{\mathcal{X}, \alpha(z)}}_{\mathcal{O}_{\mathcal{Z}, z}}(t)P^{\mathcal{O}_{\mathcal{Y}, \beta(z)}}_{\mathcal{O}_{\mathcal{Z}, z}}(t)}.
$$
This yields the inequality
\begin{equation}\label{eqth34}
     \beta_1^{\mathcal{O}_{\mathcal{X}, \alpha(z)} \times_{\mathcal{O}_{\mathcal{Z}, z}} \mathcal{O}_{\mathcal{Y}, \beta(z)}}(M)\geq  \beta_0^{\mathcal{O}_{\mathcal{X}, \alpha(z)}}(M)\beta_1^{\mathcal{O}_{\mathcal{Y}, \beta(z)}}(\mathcal{O}_{\mathcal{Z}, z})+\beta_1^{\mathcal{O}_{\mathcal{X}, \alpha(z)}}(M), 
\end{equation}
for any $\mathcal{O}_{\mathcal{X}, \alpha(z)}$-module $M$. In particular, if $M$ is the residue field, we have that
\begin{equation}\label{eqth35}
\edim(\mathcal{O}_{\mathcal{X}, \alpha(z)} \times_{\mathcal{O}_{\mathcal{Z}, z}} \mathcal{O}_{\mathcal{Y}, \beta(z)})\geq \beta_1^{\mathcal{O}_{\mathcal{Y}, \beta(z)}}(\mathcal{O}_{\mathcal{Z}, z})+\edim(X,x).
\end{equation}
Since $\edim{\mathcal{O}_{\mathcal{X}, \alpha(z)}}\geq \dim\mathcal{O}_{\mathcal{X}, \alpha(z)}$,  equations (\ref{eqth31}) and (\ref{eqth35}) provides $\beta_1^{\mathcal{O}_{\mathcal{Y}, \beta(z)}}(Z,z)=0$. But,  if $\beta_1^{\mathcal{O}_{\mathcal{Y}, \beta(z)}}(\OO_{\mathcal{Z},z})=0$, then $\OO_{\mathcal{Z},z}$ is a free $\mathcal{O}_{\mathcal{Y}, \beta(z)}$-module. The surjective map $\mathcal{O}_{\mathcal{Y}, \beta(z)}\stackrel{\beta_{z}^{\ast}}\to \OO_{\mathcal{Z},z}$  and the fact that $\OO_{Z,z}\cong\mathcal{O}_{\mathcal{Y}, \beta(z)}^{\oplus r}$, implies that $r=1$, i.e., $\mathcal{O}_{\mathcal{Y}, \beta(z)}\cong\OO_{\mathcal{Z},z}$. This yields a trivial fiber product $\mathcal{O}_{\mathcal{X}, \alpha(z)} \times_{\mathcal{O}_{\mathcal{Z}, z}} \mathcal{O}_{\mathcal{Y}, \beta(z)}$, which is a contradiction (Remark \ref{assumption}).
\end{proof}

\begin{ex}\label{exCI}
{\rm If $k$ is a field, consider the Noetherian formal schemes $\mathcal{X}=\Spf(k[[x]])$, $\mathcal{Y}=\Spf(k[[y]])$ and $\mathcal{Z}=\Spf(k)$. Let $\OO_\mathcal{X}(\mathcal{X}) \twoheadrightarrow \OO_\mathcal{Z}(\mathcal{Z})$ and $\OO_\mathcal{Y}(\mathcal{Y}) \twoheadrightarrow \OO_\mathcal{Z}(\mathcal{Z})$ be two natural morphism of  rings. Note that  $\mathcal{X}$ and $\mathcal{Y}$ are nonsingular formal schemes and  the gluing $\mathcal{X}\sqcup_{\mathcal{Z}}\mathcal{Y}=\Spf(k[[x,y]]/(xy))$  is a singular  formal scheme (see Remark \ref{Union}) (ii)). This example illustrates that even considering the most simple gluing of nonsingular formal schemes, we derive a  singular formal scheme.
}
\end{ex}

\end{document}